\documentclass[12pt, twoside, leqno]{article}

\usepackage{amsmath,amsthm}
\usepackage{amssymb}

\usepackage{enumerate}

\usepackage{graphicx}

\newtheorem{thm}{Theorem}[section]
\newtheorem{cor}[thm]{Corollary}

\newtheorem{prop}[thm]{Proposition}

\newtheorem{fact}[thm]{Fact}

\theoremstyle{definition}
\newtheorem{defin}[thm]{Definition}
\newtheorem{rem}[thm]{Remark}

\numberwithin{equation}{section}
\newcommand{\N}{{\mathbb N}}

\newcommand{\R}{{\mathbb R}}

\newcommand{\T}{{\mathbb T}^2}

\newcommand{\Cinf}{{{\mathcal C}^\infty}}

\newcommand{\vers}{\longrightarrow}

\newcommand{\Leb}{\operatorname{Leb}}

\newcommand{\Sing}{\operatorname{Sing}}

\usepackage{makeidx}

\usepackage{hyperref}

\begin{document}


\baselineskip=17pt


\title{On Physical Measures for Cherry Flows}

\author{Liviana Palmisano\\
Institute of Mathematics\\ 
Polish Academy of Sciences\\
00-956 Warszawa, Poland\\
E-mail: l.palmisano@impan.pl}

\date{}

\maketitle


\renewcommand{\thefootnote}{}

\footnote{2010 \emph{Mathematics Subject Classification}: Primary 37A99; Secondary 28D99.}

\footnote{\emph{Key words and phrases}: physical measures, Cherry flows, invariant measures.}

\renewcommand{\thefootnote}{\arabic{footnote}}
\setcounter{footnote}{0}


\begin{abstract}
Studies of the physical measures for Cherry flows were initiated in \cite{SV}. 
While the non-positive divergence case was resolved, the positive divergence one still lacked the complete description. 
Some conjectures were put forward. In this paper we contribute in this direction. 
Namely, under mild technical assumptions we solve conjectures stated in \cite{SV} by providing a description of the 
physical measures for Cherry flows in the positive divergence case.
%
%
%
%
%
%

 %
\end{abstract}

\section{Introduction}
One of the main goals of dynamical systems theory is to describe the typical behavior of orbits, especially when time goes to infinity, and understanding how this behavior is affected by small perturbations of the law that governs the system. 
\par
Such questions are especially difficult when the system is sensitive to initial conditions; that is, when a small change in the initial state results in a large variation in the long term behavior of the orbits. One way to address this problem is using the so-called physical measures. These are probability measures of a particular interest as they describe the statistical properties of a large set of orbits. 
\par
In general physical measures are still poorly understood.  Even their existence has been established only for a narrow class of systems.  
In this paper we make a contribution in this area by studying the physical measures for Cherry flows. 
\par
We recall the classical construction of a Cherry flow given in \cite{Cherry}. 
It is a $\Cinf$ flow on the two-dimensional torus without closed orbits and with two singularities, a sink and a saddle. 
In this case it is relatively easy to check that the only physical measure is the Dirac delta at the sink. 
The inverted flow, still called a Cherry flow, is a $\Cinf$ flow on the torus $\T$, without closed orbits and with two singularities, 
a saddle point and a repelling point, both hyperbolic, see Figure \ref{Cherrypic}. Describing the physical measures in this framework, which is the topic of this paper, 
is much more difficult and interesting. 
\par
Our results answer some open questions of \cite{SV}.
\par
Before we can state them formally, it is necessary to recall some basic definitions. 
\begin{figure}\label{Cherrypic}
\centering
\includegraphics[width=109mm]{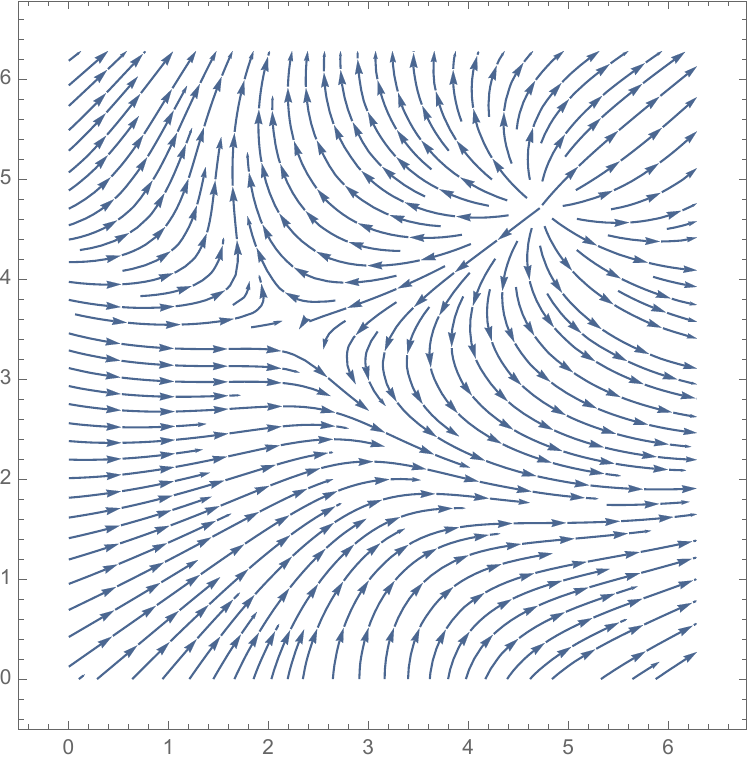}
\caption{Cherry flow}
\end{figure}

\subsection{Basic Definitions and General Remarks}
\paragraph{Physical Measures.}
Let $\phi$ be a continuous flow on a compact manifold $M$. A probability measure $\nu$ is {\it invariant} under the flow if $\nu(\phi_t(A))=\nu(A)$, for all $t\in\mathbb R$ and for any measurable set $A\subset M$. 

\begin{defin}\label{cela}
Let $t > 0$, we define the following family of probability measures $m_t(z)$, $z\in M$ by:
\begin{equation*}
\int_{M}\alpha dm_t(z)=\frac{1}{t}\int_0^t\alpha(\phi_s(z))ds,
\end{equation*} 
for each continuous function $\alpha: M \vers \mathbb R$.
\end{defin}  
\begin{defin}\label{cela1}
Let $\nu$ be an invariant probability measure. The basin of attraction $B(\nu) = B^\phi(\nu)$ of $\nu$ is the set of $z \in M$ such that:
\[
	\lim_{t\to\infty}m_t(z)=\nu \, \, \, (\text{for the weak-}\star\text{topology}).
\]
The measure $\nu$ is said to be physical if its basin of attraction has strictly positive Lebesgue measure.  
\end{defin} 

\paragraph{Cherry Flows.}
In the following we provide some definitions and properties concerning Cherry flows. We will state them in a compact form. For more details the reader can refer to \cite{ABZ}, \cite{NZE}, \cite{MvSdMM}, \cite{my}, \cite{mythesis}. 

\begin{defin}\label{Cherinvdef}
A Cherry flow is a $\Cinf$ flow on the torus $\T$ without closed orbits and with two singularities, a saddle point and a repelling point, both hyperbolic.
\end{defin}

From now on for the rest of the paper $\phi$ will denote a Cherry flow as in Definition \ref{Cherinvdef}. 
Moreover $p_s$ will denote the saddle point of $\phi$ and $p_r$ its repelling point.
\begin{prop}\label{Poinsection}
Let $\phi$ be a Cherry flow and let $\Sing(\phi)$ be the set of the singularities of $\phi$. There exists a closed $\Cinf$ curve $C$, on $\T\setminus \Sing(\phi)$ with the following properties:
\begin{itemize}
\item $C$ is everywhere transversal to the flow;
\item $C$ is not contractable to a point.
\end{itemize}
\end{prop}
\begin{defin}
The closed curve $C$ constructed in Proposition \ref{Poinsection} is called a closed transversal.
\end{defin}
\begin{fact}
Every Cherry flow admits a closed transversal $C$. Notice that $T^2\setminus C$ is $C^{\infty}$-equivalent to an annulus $\mathbb S^1\times(0,1)$ and we can write $T^2\cong\mathbb S^1\times\left[0,1\right]/\sim$, where $\left(s,0\right)\sim\left(s,1\right)$. Consider $\phi$ as a flow on $T^2\cong\mathbb S^1\times\left[0,1\right]$ where we identify $\mathbb S^1\times\left\{0\right\}$ and $\mathbb S^1\times\left\{1\right\}$. After this change of coordinates, the resulting flow is a Cherry flow.
\end{fact}

Let now $g$ be the first return map of the flow $\phi$ to the closed transversal. 
The existence of $g$ is guaranteed by Theorem 2.6.1 in \cite{NZE} and by the absence of closed trajectories for Cherry flows.
Observe that $g$ is $\Cinf$ everywhere except at one point which belongs to the stable manifold of the saddle point and which we will assume to be zero (we identify $\mathbb S^1$ with $\left[-\frac{1}{2},\frac{1}{2}\right]_{-\frac{1}{2}\sim\frac{1}{2}}$). We denote by $a$ and $b$ respectively the left-sided limit and the right-sided limit of the orbit of the discontinuity point $0$ and by $U=(a,b)$.
\par
We consider now the flow $\varphi$ obtained by reversing the direction of $\phi$. The repelling point of $\phi$ becomes an attractive point for $\varphi$ which is then a Cherry flow like in Cherry's example (\cite{Cherry}). In this case, the first return map $f$ of $\varphi$ 
to the closed transversal is a circle endomorphism which is $\Cinf$ everywhere with the exception of the points $a$ and $b$ where it is continuous and it is constant on the interval $U=(a,b)$. Moreover, after a change of coordinates, on a half open neighborhood of these two points $f$ can be written as $x^{\frac{\lambda_1}{-\lambda_2}}$ where  $\lambda_1>0>\lambda_2$ are the eigenvalues of the saddle point $p_s$ of $\phi$. The respective formula for $g$ is clearly $x^{\frac{-\lambda_2}{\lambda_1}}$. 

\paragraph{Rotation Number.}
As $f$ is a monotone circle map, it has a rotation number; this number is the quantity which measures the rate at which an orbit winds around the circle. More precisely, if $F$ is a lifting of $f$ to the real line, the rotation number of $f$ is the limit 
\begin{displaymath}
\rho(f)=\lim_{n\to\infty}\frac{F^n(x)}{n} \textrm{ (mod $1$)}.
\end{displaymath}
This limit exists for every $x$ and its value is independent of $x$. 
\par
We can then define the rotation number of any flow $\varphi$ obtained by reversing the direction of a Cherry flow $\phi$ as follow:
\begin{defin}
The rotation number of $\varphi$ is the rotation number of its first return map to any closed transversal.
\end{defin}
\begin{fact}\label{rotCf}
It is easy to check that the rotation number $\rho$ of $\varphi$ does not depend on the choice of the closed transversal.
\end{fact}
Consequently the rotation number of any Cherry flow is defined:
\begin{defin}
Let $\phi$ be a Cherry flow as in Definition \ref{Cherinvdef}. The rotation number of $\phi$ is the rotation number of the flow $\varphi$ obtained by reversing the direction of $\phi$. 
\end{defin}

Since the considered flow does not have closed orbits, then $f$ has an irrational rotation number $\rho$ which admits an expansion as an infinite continued fraction
\begin{displaymath}
\rho=\frac{1}{a_1+\frac{1}{a_2+\frac{1}{\cdots}}},
\end{displaymath}
where $a_i$ are positive integers.

If we cut off the portion of the continued fraction beyond the $n$-th position, and write the resulting fraction in lowest terms as $\frac{p_n}{q_n}$ then the numbers $q_n$ for $n\geq 1$ satisfy the recurrence relation
\begin{equation}\label{q_n}
q_{n+1} = a_{n+1}q_n+q_{n-1};\textrm{  } q_0 = 1;\textrm{  } q_1 = a_1.
\end{equation}
The number $q_n$ is the number of times we have to iterate the rotation by $\rho$ in order that the orbit of any point makes its closest return so far to the point itself (see Chapter I, Sect. I in \cite{deMvS}).

\begin{defin}\label{boundedtype}
Let $\rho$ be an irrational number and let $(a_i)_{i\in\N}$ be the integers of the representation as an infinite continued fraction of $\rho$. We say that $\rho$ is of bounded type if and only if there exist a positive real number $M$ such that, for any $i$, $a_i < M$.
\end{defin}

\paragraph{First Return Time.}
\begin{defin}
Let $z$ be a point of the closed transversal. The first return time $\tau(z)$ for $\phi$ to $\mathbb S^1$ is the number of iteration of $g$ needed by $z$ to come back to $\mathbb S^1$ for the first time.
\end{defin}
\begin{fact}\label{lem:taulog}
The first return time $\tau(z)$ for $\phi$ to $\mathbb S^1$ has a logarithmic singularity in $0$. This means that for all $\epsilon > 0$, there exists a constant $C>0$ such that, for all $z$ in the interval $\left(-\epsilon,\epsilon\right)$, we have that $\frac{1}{C}\leq\frac{\tau(z)}{-\log|z|}\leq C$. In other word $\tau(z)$ is of order of $-\log|z|$.
\end{fact}

Since $f$ preserves the order and it does not have periodic points, by Poincar\'e's Theorem, there exists a continuous order preserving degree one function $h:\mathbb{S}^{1}\rightarrow\mathbb{S}^{1}$ such that $h\circ f=R_{\rho}\circ h$, where $R_{\rho}$ is the rotation by $\rho$.  
Then the probability measure $\mu=h^*(\Leb)$  supported on the minimal set is well defined and it is the only invariant ergodic measure for $f$.
\par 

By Proposition $2$ in \cite{SSV} this measure $\mu$ can be extended to a invariant probability measure $\nu$ on the torus supported on the quasi-minimal set if $\int_{\mathbb S^1}\tau d\mu$ is convergent.
\begin{figure}\label{Cherrypic}
\centering
\includegraphics[width=109mm]{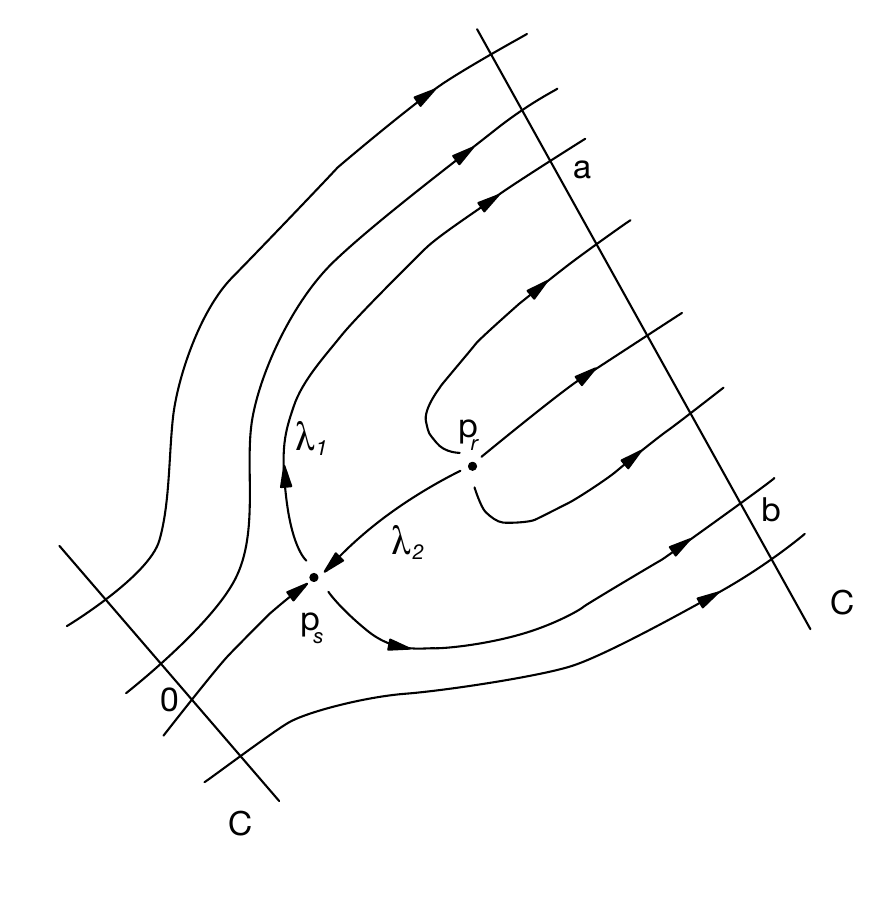}
\caption{Cherry flow}
\end{figure}
\subsection{Discussion and Statement of the Results}
In this paper we are interested in the physical measures for Cherry flows which are $\Cinf$ flows on the torus $\T$, 
without closed orbits and with two singularities, a saddle point $p_s$ and a repelling point $p_r$, both hyperbolic (see Figure \ref{Cherrypic}). 
\par 
This problem was studied firstly in \cite{SV} in which the authors gave a description of the physical measures for some class of Cherry flows. They discovered that this problem is related to the variation of the divergence of the flow at the saddle point.\footnote{
Observe that the divergence of a Cherry flow $\phi$ at the saddle point $p_s$ is exactly the sum of the eigenvalues of the flow $\phi$ at $p_s$.}
\par
To be more precise, let $\lambda_1>0>\lambda_2$ be the two eigenvalues at the saddle point. In the non-positive divergence case when $\lambda_1\leq-\lambda_2$, \cite{SV} shows that the Dirac deltas at the singularities are the only ergodic invariant probability measures. Moreover \cite{SV} establishes that the Dirac delta at the saddle point is the physical measure for the flow. 
\par
On the other hand in positive divergence case, $\lambda_1>-\lambda_2$, in addition to the Dirac deltas at the singularities, there exists another ergodic invariant probability measure $\nu$, which is supported on the quasi-minimal set of the flow and it is different from the Dirac delta at the saddle point. Under additional assumptions of strictly positive divergence, $\lambda_1>-2\lambda_2$, and of the rotation number of bounded type, they prove that $\nu$ is physical and they put forward a conjecture that this holds for any $\lambda_1>-\lambda_2$. 
\par
In this paper we present some new results in this direction. \\

\par

If the divergence at the saddle point is positive, under the hypothesis that the rotation number is of bounded type, we have: 
\begin{thm}\label{M1}
Let $\phi$ be a Cherry flow with eigenvalues at the saddle point $\lambda_1>0>\lambda_2$. If $\lambda_1>-\lambda_2$ and $\phi$ has rotation number of bounded type, the ergodic invariant probability measure $\nu$ supported on the quasi-minimal set is the physical measure for $\phi$ with attraction basin having full Lebesgue measure. 
\end{thm}
If the divergence at the saddle point becomes strictly positive, {\it without any assumption} on the rotation number, we have: 
\begin{thm}\label{M2}
Let $\phi$ be a Cherry flow with eigenvalues at the saddle point $\lambda_1>0>\lambda_2$. If $\lambda_1\geq-3\lambda_2$, the ergodic invariant probability measure $\nu$ supported on the quasi-minimal set is the physical measure for $\phi$ with attraction basin having full Lebesgue measure. 
\end{thm}
Theorem \ref{M1} together with results of \cite{SV} provides a complete description of the physical measures for Cherry flows having rotation number of bounded type (bounded regime). In the unbounded regime the case $1<\frac{\lambda_1}{-\lambda_2}<3$ remains still open. We will comment on technical problems arising in this case in Remark \ref{finalremark}.

\subsection{Standing Assumptions and Notations}

\vspace{0.3cm}
Let $0$ be the discontinuity point of $g$. In order to simplify the notation we shall write:
\begin{itemize}
\item $\underline{i}=f^{i}(0)$,
\item $\underline{i}_R=R_{\rho}^i(0)$.

\end{itemize}
We observe that, because of the properties of $f$, underlined non-positive integers of the type $\underline{-i}$ represent intervals.
\paragraph {Distance between Points.}
We denote by $\left(a,b\right)=\left(b,a\right)$ the  shortest  open interval between $a$ and $b$ regardless of the order of these two points. The length of that interval in the natural metric on the circle will be denoted by $\left|a-b\right|$. Following \cite{5aut}, let us adopt these notational conventions for the distance between the preimages of the first return function $f$:
\begin{itemize}
\item $\left|\underline{-i}\right|$ stands for the length of the interval $\underline{-i}$.
\item Consider a point $x$ and an interval $\underline{-i}$ not containing it. Then the distance from $x$ to the closest endpoint of $\underline{-i}$ will be denoted by $\left|\left(x,\underline{-i}\right)\right|$, and the distance to the most distant endpoint by $\left|\left(x,\underline{-i}\right]\right|$.
\item We define the distance between the endpoints of two intervals $\underline{-i}$ and $\underline{-j}$ analogously. For example, $\left|\left(\underline{-i},\underline{-j}\right)\right|$ denotes the distance between the closest endpoints of these two intervals while $\left|\left[\underline{-i},\underline{-j}\right)\right|$ stands for $\left|\underline{-i}\right|+\left|\left(\underline{-i},\underline{-j}\right)\right|$.
\end{itemize}

\section{Proof of Theorem \ref{M1}}
We consider the sequence
\begin{displaymath}
\alpha_n=\frac{\left|(\underline{-q_n},\underline{0})\right|}{\left|[\underline{-q_n},\underline{0})\right|}
\end{displaymath} 
and we prove the following proposition:
\begin{prop}\label{amo}
Let $\lambda_1>0>\lambda_2$ be the eigenvalues of the saddle point of $\phi$ and let $f$ be the first return function of the reversed flow $\varphi$. If $f$ has rotation number of bounded type and $\frac{\lambda_1}{-\lambda_2}\in(1,2]$, then there exist two constants $K>0$ and $C<1$ such that for $n$ big enough, $\frac{-\log\alpha_n}{q_{n+1}}\leq K C^n$. This constant $C$ doesn't depend  on the eigenvalues $\lambda_1$ and $\lambda_2$ of the saddle point.
\end{prop}
\begin{proof}
We write $\ell=\frac{\lambda_1}{-\lambda_2}$. Before beginning the proof  it is necessary  to recall that:
\begin{enumerate}
\item by the recursive formula (\ref{q_n}) $q_0 = 1$, $q_1 = a_1$ and $q_{n+1}=a_{n+1}q_{n}+q_{n-1}$,
\item by Proposition $6$ in \cite{5aut}, for $n$ big enough $\alpha_n\geq K_1\alpha_{n-1}^{\frac{1-\ell^{-a_{n+1}}}{\ell-1}}\alpha_{n-2}^{\ell^{-a_n}}$ where $K_1$ is a positive constant.
\end{enumerate}
In order to simplify the proof we will work assuming that, for $n$ big enough, 
\begin{equation}\label{senk}
\alpha_n\geq\alpha_{n-1}^{\frac{1-\ell^{-a_{n+1}}}{\ell-1}}\alpha_{n-2}^{\ell^{-a_n}}.
\end{equation}
This hypothesis is not restrictive. Indeed, we observe that for every $n$ we have $q_n\geq\frac{\beta^n}{K_2}$ with $\beta=\frac{1+\sqrt{5}}{2}$ and $K_2$ a positive constant. Then the estimation remains true also in the general case 
as one could to each inequalities below the term $\frac{\log(K_1)}{q_n}$ which trivially satisfies the desired estimation. 
\par
To lighten the notation, we introduce a new sequence $(\theta_n)_{n \in \N}$ defined for all $n$, by $\theta_n:=-\log\alpha_n$.
\par
We fix $n_0\in\N$ such that, for all $n\geq n_0$, (\ref{senk}) is satisfied.
\par
We shall prove the proposition by induction on $n$. We take $C=C(\ell)=\sup_i\left(\frac{1-\ell^{-a_{i}}}{(\ell-1)a_i}\right)^{\frac{1}{n_0}}$ and $K\geq\max\lbrace{\theta_{n_0-2},\theta_{n_0-1}\rbrace}$. 
We observe that, for all $1<\ell\leq2$, we have $C<1$; if we consider $C$ as a function of $\ell$, then, in the interval $(1,2]$, $C(\ell)$ is continuous, decreasing and moreover $\lim_{\ell\to 1}C(\ell)=1$ and $C(2)<1$. 
\par 
We observe that, for any natural number $i\geq 1$
\begin{equation}\label{sl2}
\ell^{-a_i}\leq \frac{1-\ell^{-a_{i}}}{(\ell-1)a_i}\leq C^{n_0}\leq C.
\end{equation}
\par
We now begin the proof by induction.
\begin{itemize}
\item Let $n_0$ be as above. By (\ref{senk}) and (\ref{sl2}), we have that:
\begin{eqnarray*}
\theta_{n_0}&\leq&\left(\frac{1-\ell^{-a_{n_0+1}}}{\ell-1}\right)\theta_{n_0-1}+\ell^{-a_{n_0}}\theta_{n_0-2}\\
&\leq& C^{n_0}K a_{n_0+1}+ C^{n_0}K\\
&\leq & K C^{n_0}\left(a_{n_0+1}+1\right).
\end{eqnarray*}
By point $(1)$~:
\begin{eqnarray}\label{n=1}
\theta_{n_0}\leq K C^{n_0}q_{n_0+1}. 
\end{eqnarray}
\item We now prove the assertion for $n_0+1$. By (\ref{senk}), (\ref{n=1}) and (\ref{sl2}):
\begin{eqnarray*}
\theta_{n_0+1}&\leq&\left(\frac{1-\ell^{-a_{n_0+2}}}{\ell-1}\right)\theta_{n_0}+\ell^{-a_{n_0+1}}\theta_{n_0-1}
\\
&\leq& K C^{n_0}\left(\frac{1-\ell^{-a_{n_0+2}}}{(\ell-1)a_{n_0+2}}\right)a_{n_0+2}q_{n_0+1}+ K C^{n_0}.
\end{eqnarray*}
And by point $(1)$ and (\ref{sl2})~:
\begin{eqnarray*}
\theta_{n_0+1}&\leq&  K C^{n_0+1}\left(a_{n_0+2}q_{n_0+1}+\frac{C^{n_0}}{C}\right)\\&\leq& K C^{n_0+1}q_{n_0+2}.
\end{eqnarray*}
\item We now assume that the assertion is true for $n-2$ and for $n-1$ and we prove it for $n$. 

By (\ref{senk}) and the inductive hypothesis we have that:
\begin{eqnarray*}
\theta_n&\leq&\left(\frac{1-\ell^{-a_{n+1}}}{\ell-1}\right)\theta_{n-1} +\ell^{-a_n}\theta_{n-2}\\&\leq& K \left(\frac{(1-\ell^{-a_{n+1}})}{(\ell-1)a_{n+1}} C^{n-1} a_{n+1} q_n+C^{n-2}\ell^{-a_n}q_{n-1}\right).
\end{eqnarray*}
Finally, by \ref{sl2} and by point $(1)$
\begin{eqnarray*}
\theta_n&\leq& K C^n\left( a_{n+1} q_n+\frac{\ell^{-a_n}}{C^2}q_{n-1}\right)\\&\leq &K C^n q_{n+1}.
\end{eqnarray*}
\end{itemize}
\par

So, the assertion of the lemma is true for all $n\in\mathbb N$ big enough.
 \end{proof}

As a direct consequence of Proposition \ref{amo} we have the following Corollary:
\begin{cor}\label{coramo}
Let $\lambda_1>0>\lambda_2$ be the eigenvalues of the saddle point of $\phi$ and let $f$ be the first return function of the reversing flow $\varphi$. If $f$ has rotation number of bounded type and $\frac{\lambda_1}{-\lambda_2}\in(1,2]$, then there exist two constants $K>0$ and $C<1$ such that for $n$ big enough, $\frac{-\log|(\underline{q_{n}},\underline{0})|}{q_{n+1}}\leq K C^n$. This constant $C$ doesn't depend  on the eigenvalues $\lambda_1$ and $\lambda_2$ of the saddle point.
\end{cor} 
We recall the following theorem proved in \cite{5aut}~:
\begin{thm} \label{wp}
Let $\lambda_1>0>\lambda_2$ be the eigenvalues of the saddle point of $\phi$ and let $f$ be the first return function of the reversing flow $\varphi$. If $\lambda_1>-\lambda_2$, then $\cup_{i=0}^{\infty}f^{-i}(U)$ has full Lebesgue measure on $\mathbb S^1$.
\end{thm}
The proof of Theorem \ref{M1} uses the main ideas of the proof of Theorem $3$ in \cite{SV}. 
\begin{proof}[Proof of Theorem \ref{M1}]
\par
By Theorem $2$ in \cite{SV} we know that the flow $\phi$ has an invariant probability measure $\nu$ supported on the quasi-minimal set  which corresponds to the extension of the $f$-invariant  measure $\mu$ (defined by $\mu=h^{\star}(\Leb)$ ). It remains to prove that $\nu$ is a physical measure for $\phi$ and that its basin of attraction has full Lebesgue measure.  
\par
By Theorem \ref{wp}, it is sufficient to prove that the points of the wandering set of $\varphi$ are in the basin of attraction of $\nu$. Since all points of the wandering set pass through the flat interval of $f$, then we have just to prove that any point of $U$ is in the basin of attraction of $\nu$. 
\par
Let $z\in U$, $\underline{n}_g=g^{n-1}(z)$ and $t_n=\tau(\underline{n}_g)$. For all $t>0$ there exists $N\in\mathbb N$ such that $t=t_1+t_2+\dots+t_N+\tilde{t}$ where $0<\tilde{t}\leq t_{N+1}$. Moreover, let $n\in\mathbb N$ such that $q_n\leq N<q_{n+1}$. Since $\tau$ is uniformly bounded below we have that:
\begin{equation}\label{eq:tau}
t\geq C N 
\end{equation}
with $C$ a positive constant.
\par
Let $m_t$ be the probability measure introduced in Definition \ref{cela}).
Since the only invariant probability measures are $\delta_s$, $\delta_r$ and $\nu$ (for more details see Theorem 2 in \cite{SV}) and since $p_r$ is repelling, then the limits for $m_t$ will be of the form
\begin{equation}\label{eq:limmt}
\gamma\delta_s+(1- \gamma)\nu, 
\end{equation}
for some $\gamma\in\left[0,1\right]$. In order to prove that $z$ is in the basin of attraction of $\nu$, which means that $\lim_{t\to\infty}m_t=\nu$, it is necessary to prove that $\gamma=0$.
\par
We fix $0<n_0<n$ and we prove that the trajectory of $z$ under $\phi$ spends the most of the time outside of
\begin{equation*} 
A_{n_0}=\lbrace{\phi_s(w)~: w\in(\underline{q_{n_0}},\underline{q_{n_0+1}}), 0\leq s\leq\tau(w)\rbrace}. 
\end{equation*}
We prove that, choosing correctly $n_0$, the time $t_{A_{n_0}}$  which the trajectory $\phi_s(z)$, $0\leq s\leq t$ spends in $A_{n_0}$ can be done arbitrarily small in comparison to $t$, for all $t$ big enough. 
\par
In order to do this we divide $A_{n_0}$ and we start to estimate the time $t_{B_l}$ spent by the trajectory $\phi_s(z)$, $0\leq s\leq t$ in
\begin{equation*}
B_{l}=\lbrace{\phi_s(w)~: w\in(\underline{q_{l}},\underline{q_{l+2}}), 0\leq s\leq\tau(w)\rbrace}.
\end{equation*}
\par
We observe that, since $f$ is the first return function of the flow obtained reversing the direction of $\phi$, if $h$ is the semi-conjugation between $f$ and the rotation $R_{\rho}$, then 
\begin{equation*} 
h(\underline{n}_g)=h(g^{n-1}(z))=h(f^{-n+1}(z))=h(f^{-n}(0))=\underline{-n}_R.
\end{equation*} 
Then for all $l\in\N$ we have $\underline{q_l}_g\in(\underline{q_{l-1}},\underline{q_{l+1}})$ and $\underline{q_l}\in(\underline{q_{l-1}}_g,\underline{q_{l+1}}_g)$. So we can say that the number of the points $\underline{i}_g$, $1\leq i\leq N$ in $(\underline{q_{l}},\underline{q_{l+2}})$ is equal to the number of the points $\underline{-i}_R$, $1\leq i\leq N$ in $(\underline{q_{l}}_R,\underline{q_{l+2}}_R)$. 
\par
So we estimate the number $N_l$ of the points $\underline{-i}_R$, $1\leq i\leq N$ which are in $(\underline{q_{l}}_R,0)$. Since $|(\underline{q_l}_R,{0})|$ is of the order of $\frac{1}{q_{l+1}}$ and since the rotation is a bijection preserving the distance, we can divide the circle in exactly $q_{l+1}$  disjoints  images of $(\underline{q_l}_R,{0})$ and any image has $N_l$ points  $\underline{-i}_R$, $1\leq i\leq N$.
\par
In conclusion  $$q_{l+1}N_l\leq N$$

and the number of the points $\underline{-i}_R$, $1\leq i\leq N$ which are in $(\underline{q_{l}}_R,\underline{q_{l+2}}_R)$ is less than or equal to $\frac{N}{q_{l+1}}$.
\par
By Fact \ref{lem:taulog}, Equation (\ref{eq:tau}) and Corollary \ref{coramo} we have that:
\begin{eqnarray*}
 \frac{t_{A_{n_0}}}{t}=\frac{1}{t}\sum_{l=n_0}^{n-1}t_{B_l}&\leq &\frac{C_3N}{t}\sum_{l=n_0}^{n-1}\frac{-\log|(\underline{q_{l+2}},\underline{0})|}{q_{l+1}}\\
 &\leq&\frac{C_3 }{C}\sum_{l=n_0}^{n-1}\frac{-\log|(\underline{q_{l+2}},\underline{0})|}{q_{l+1}}\\
 &\leq&\frac{C_3}{C C_4}\sum_{l=n_0}^{n-1}(C_5)^{l+2}.
\end{eqnarray*}
Observe that we are assuming a supplementary hypothesis on the eigenvalues $\lambda_1>0>\lambda_2$ of the saddle point: $\lambda_1\leq-2\lambda_2$. The case $\lambda_1>-2\lambda_2$ is proved in \cite{SV} (see Theorem $3$).
Finally, since $\sum_{l=n_0}^{\infty}(C_5)^l$ is convergent, taking $n_0$ big enough, we can make $\frac{t_{A_{n_0}}}{t}$ as small as we want.  
\par
We observe that we have the same result if in place of $A_{n_0}$ we consider $A_{n_0-c}$ with $c>0$ and $A_{n_0}\Subset A_{n_0-c}$.

\par
 It remains to prove that if $\lim_{n_0\to\infty}\frac{t_{A_{n_0}}}{t}=0$ then $\gamma=0$.
 \par
 We suppose by contradiction that $\gamma>0$ and we recall that there exists a sequence strictly non-decreasing sequence of positive reals $(t_n)_{n \in \N}, \, t_n \vers +\infty$ when $n \to +\infty$, such that: $\lim_{t_n\to\infty}m_{t_n}(z)=\gamma\delta_s+(1- \gamma)\nu$ (see (\ref{eq:limmt})).
 \par

 Let us fix $\epsilon>0$. So there exists $T>0$ such that for all $n \in \N$ for which $t_n>T$ and for any continuous $\alpha:\T\vers\R$
\begin{equation}\label{india}
 \left|\int_{\T}\alpha dm_{t_n}(z)-\int_{\T}\alpha d(\gamma\delta_s+(1- \gamma)\nu)\right|<\epsilon.
\end{equation}
Now let $c>0$ be such that $A_{n_0}\Subset A_{n_0-c}$. Let $\alpha$ be a bump function  with compact support such that, for all $x\in A_{n_0}$, $\alpha(x)=1$ and for all $x\in (A_{n_0-c})^c$, $\alpha(x)=0$. We observe that
\begin{equation*}
\frac{t_{A_{n_0}}}{t_n}\leq\int_{\T}\alpha dm_{t_n}(z)\leq \frac{t_{A_{n_0-c}}}{t_n}
\end{equation*}
and we recall that by hypothesis $\lim_{n_0\to\infty}\frac{t_{A_{n_0}}}{t_n}=\lim_{n_0\to\infty}\frac{t_{A_{n_0-c}}}{t_n}=0$.
\par
Then, for $n$ big enough, we deduce by \eqref{india},
\begin{equation}\label{eq:ro}
 \gamma-\epsilon <\gamma+(1-\gamma)\nu(A_{n_0})-\epsilon<\epsilon
\end{equation}
which is in contradiction with the hypothesis that $\gamma>0$.
\par
So $\gamma=0$ and (for the weak-$\star$ topology) $\lim_{t\to\infty}m_t(z)=\nu$. By Definition \ref{cela1}, $z$ is in the basin of attraction of $\phi$.
\end{proof}

\section{Proof of Theorem \ref{M2}} 
The idea of the proof of this theorem is similar to the proof of Theorem \ref{M1}. The main technical tool is the fact that, under the condition $\lambda_1\geq-3\lambda_2$, without any assumption on the rotation number, the sequence $\frac{|(\underline{0}, \underline{q_{n}})|}{|(\underline{0}, \underline{q_{n-2}})|}$ is bounded away from zero (the second claim of Theorem 1.2 in \cite{my2}).

\begin{proof}[Proof of Theorem \ref{M2}]
By the second claim of Theorem 1.2 in \cite{my2} we can assume that there exist $n_0\in\mathbb N$ and a constant $\alpha\in\left(0,1\right)$ such that $\frac{|(\underline{0}, \underline{q_{n}})|}{|(\underline{0}, \underline{q_{n-2}})|}>\alpha^2$ for $n\geq n_0>0$.  Then, by induction
\begin{equation}\label{eq:ricsuan}
|(\underline{0}, \underline{q_{n}})|>C\alpha^n
\end{equation}
for some $C>0$. 
\par
By Theorem $2$ in \cite{SV}, there exists an invariant probability measure $\nu$ supported on the quasi-minimal set. We prove that the basin of attraction of $\nu$ has full Lebesgue measure; so $\nu$ is a physical measure for $\phi$.
\par
Like in the proof of Theorem \ref{M1} we prove that any point of $U$ is in the basin of attraction of $\nu$.
\par
Let $z\in U$, $\underline{n}_g=g^{n-1}(z)$ and $t_n=\tau(\underline{n}_g)$. For all $t>0$ there exists $N\in\mathbb N$ such that  $t=t_1+t_2+\dots+t_N+\tilde{t}$ where $0<\tilde{t}\leq t_{N+1}$ and there exists $n\in\mathbb N$ such that $q_n\leq N<q_{n+1}$. Since $\tau$ is uniformly bounded below, we have that 
\begin{equation}\label{eq:tau1} 
t\geq C_1N
\end{equation}
 with $C_1 > 0$ a positive constant.
\par
Let $m_t$ be the probability measure as introduced in Definition \ref{cela}). The possible limits for $m_t$ must have the form $\gamma\delta_s+(1- \gamma)\nu$, for some $\gamma\in\left[0,1\right]$. We have to prove that $\gamma$ is zero (for the details see the proof of Theorem \ref{M1}).
\par
We fix $0<n_0<n$ and we prove that the orbit of $z$ under $\phi$ spends most of time outside of
\begin{equation*}
A_{n_0}=\lbrace{\phi_s(w)~: w\in(\underline{q_{n_0}},\underline{q_{n_0+1}}), 0\leq s\leq\tau(w)\rbrace}.
\end{equation*}
\par
The time $t_{A_{n_0}}$  spent in $A_{n_0}$ will be calculated as the sum of the times $t_{B_l}$ spent in small pieces of $A_{n_0}$ of the form
\begin{equation*}
B_{l}=\lbrace{\phi_s(w)~: w\in(\underline{q_{l}},\underline{q_{l+2}}), 0\leq s\leq\tau(w)\rbrace}.
\end{equation*}
For these reasons, it is necessary to estimate the number of the points $\underline{i}_g$, $1\leq i\leq N$ in $(\underline{q_{l}},\underline{q_{l+2}})$ which, like in Theorem \ref{M1}, coincides with the number of the points $\underline{-i}_R$, $1\leq i\leq N$ in $(\underline{q_{l}}_R,\underline{q_{l+2}}_R)$ 
which is less than or equal to $\frac{N}{q_{l+1}}$.
\par
Finally by  (\ref{eq:ricsuan}), (\ref{eq:tau1}), Fact \ref{lem:taulog} and the fact that $q_{l}\geq\frac{\beta^l}{C_6}$ for $\beta=\frac{1+\sqrt{5}}{2}$ we have that:
\begin{eqnarray*}
 \frac{t_{A_{n_0}}}{t}=\frac{1}{t}\sum_{l=n_0}^{n-1}t_{B_l}&\leq &\frac{C_5N}{t}\sum_{l=n_0}^{n-1}\frac{-\log|(\underline{q_{l+2}},\underline{0})|}{q_{l+1}}\\
 &\leq&\frac{C_5}{C_1}\sum_{l=n_0}^{n-1}\frac{-\log|(\underline{q_{l+2}},\underline{0})|}{q_{l+1}}\\
 &\leq&\frac{C_5C\alpha}{C_1}\sum_{l=n_0}^{n-1}\frac{l}{q_{l+1}}\\
 &\leq& \frac{C_5C\alpha C_6}{C_1}\sum_{l=n_0}^{n-1}\frac{l}{\beta^{l+1}}.
\end{eqnarray*}
\par
In conclusion, taking $n$ big enough, we can make $\frac{t_{A_{n}}}{t}$ as small as we want: then, like in the proof of Theorem \ref{M1}, $\gamma=0$ and $z$ is in the basin of attraction of $\nu$. 
\end{proof}

\begin{rem} \label{finalremark}
The proof of Theorem \ref{M1} hinges on the recursive estimate (\ref{senk}) for $\alpha_n$ (used to demonstrate Proposition \ref{amo}). In the case of rotation number of bounded type such an estimate was found in \cite{5aut} and is sufficient to conduct the proof. The case of unbounded type is more problematic as the estimate, provided by \cite{my2}, is not sufficient any more. To circumvent this problem we assume additionally that  the ratio of the eigenvalues at the saddle point is greater than or equal to $3$, which assures that the sequence $\alpha_n$  is bounded away from zero  (Theorem \ref{M2}). We note that an improved estimate may lead to the complete description of the physical measures for Cherry flows without any assumption on the rotation number.
\end{rem}

\subsection*{Acknowledgments}
I would sincerely thank Prof. J. Graczyk for introducing me to the subject of this paper, his valuable advice and continuous encouragement. I am also very grateful to Prof. R. Saghin and Prof. E. Vargas for their willingness to answer to all my questions. This paper was partially supported by funds allocated to the implementation of the international co-funded project in the years 2014-2018, 3038/7.PR/2014/2, and by the EU grant PCOFUND-GA-2012-600415.

\end{document}